\documentclass[oneside]{amsart}
\usepackage{amsmath,amsfonts,amssymb}%
\usepackage[all]{xy}
\usepackage[usenames]{color}
\usepackage{setspace}
\usepackage{comment}
\usepackage{url}

\usepackage{amsmath, amsthm, amsfonts, amssymb, stmaryrd}%
\usepackage{mathrsfs}

\usepackage{enumerate}

\usepackage{xparse, expl3}

\usepackage[capitalize]{cleveref}  
\crefname{lemma}{Lemma}{Lemmas}
\crefname{corollary}{Corollary}{Corollaries}
\crefname{theorem}{Theorem}{Theorems}
\crefname{equation}{Equation}{Equations}
\crefname{example}{Example}{Examples}
\crefname{section}{Section}{Sections}
\crefname{subsection}{Section}{Sections}

\newcommand{\defn}[1]{{\bf{#1}}}

\def\Powerset{\mathcal{P}}

\newcommand{\Erdos}{\text{Erd\H{o}s}}

\DeclareMathOperator{\dom}{dom}
\DeclareMathOperator{\range}{ran}

\DeclareMathOperator{\Rel}{Rel}
\DeclareMathOperator{\arity}{ar}

\DeclareMathOperator{\Turingop}{T}

\newcommand{\TuringEq}{\EM{\equiv_{\Turingop}}}
\newcommand{\TuringLeq}{\EM{\leq_{\Turingop}}}
\DeclareDocumentCommand{\TuringJump}{d[]}
{
\IfNoValueTF{#1}
	{
	\EM{\mbf{0'}}
	}
	{
	\EM{\mbf{#1'}}
	}
}

\DeclareMathOperator{\Sunflowerop}{Sun}
\DeclareDocumentCommand{\SunflowerComp}{d[]}
{
\IfNoValueTF{#1}
	{
	\EM{\Sunflowerop}
	}
	{
	\EM{\Sunflowerop(#1)}
	}
}

\DeclareMathOperator{\FinSetCompop}{FS}
\DeclareDocumentCommand{\FinSetComp}{d[]}
{
\IfNoValueTF{#1}
	{
	\EM{\FinSetCompop}
	}
	{
	\EM{\FinSetCompop(#1)}
	}
}

\DeclareDocumentCommand{\FinSetSeqComp}{d[]}
{
\IfNoValueTF{#1}
	{
	\EM{\FinSetCompop_{<\w}}
	}
	{
	\EM{\FinSetCompop_{<\w}(#1)}
	}
}

\DeclareDocumentCommand{\dual}{d()}
{
\IfNoValueTF{#1}
	{
	\EM{\hat{\ }}
	}
	{
	\EM{\hat{#1}}
	}
}

\DeclareDocumentCommand{\SizeAS}{d()}
{
\IfNoValueTF{#1}
	{
	\EM{|\cdot|}
	}
	{
	\EM{|#1|}
	}
}

\DeclareDocumentCommand{\compcK}{d[]}
{
\IfNoValueTF{#1}
	{
	\EM{\mathbb{K}}
	}
	{
	\EM{\mathbb{K}[#1]}
	}
}

\DeclareDocumentCommand{\AgeK}{d[]}
{
\IfNoValueTF{#1}
	{
	\EM{\mathbf{K}}
	}
	{
	\EM{\mathbf{K}[#1]}
	}
}

\DeclareDocumentCommand{\Rel}{d[]}
{
\IfNoValueTF{#1}
	{
	\EM{\mathcal{R}}
	}
	{
	\EM{\mathcal{R}_{#1}}
	}
}

\DeclareDocumentCommand{\Func}{d[]}
{
\IfNoValueTF{#1}
	{
	\EM{\mathcal{F}}
	}
	{
	\EM{\mathcal{F}_{#1}}
	}
}

\DeclareDocumentCommand{\ar}{d[]}
{
\IfNoValueTF{#1}
	{
	\EM{\arity}
	}
	{
	\EM{\arity_{#1}}
	}
}

\DeclareDocumentCommand{\Fn}{d<> d[] d()}
{
\IfNoValueTF{#3}
	{
	\EM{\textrm{Fn}(#1, #2)}
	}
	{
	\EM{\textrm{Fn}(#1, #2, #3)}
	}
}

\DeclareMathOperator{\Sym}{Sym}
\DeclareDocumentCommand{\Perm}{d()}
{
\EM{\Sym(#1)}
}

\DeclareDocumentCommand{\tdcl}{d[] d()}
{
\IfNoValueTF{#1}
{
    \EM{\textbf{term}(#2)}
}
{
    \EM{\textbf{term}_{#1}(#2)}
}
}

\DeclareDocumentCommand{\gdcl}{d[] d()}
{
\IfNoValueTF{#1}
{
    \EM{\textrm{gcl}(#2)}
}
{
    \EM{\textrm{gcl}_{#1}(#2)}
}
}

\DeclareDocumentCommand{\qftp}{d[] d()}
{
\IfNoValueTF{#1}
{
    \EM{\quantfreetp(#2)}
}
{
    \EM{\quantfreetp_{#1}(#2)}
}
}

\DeclareMathOperator{\quantfreetp}{qtp}

\DeclareDocumentCommand{\Closure}{d<> d()}
{
\IfNoValueTF{#1}
{
    \EM{\textrm{cl}(#2)}
}
{
    \EM{\textrm{cl}_{#1}(#2)}
}
}

\DeclareDocumentCommand{\ClosureMap}{d[] d()}
{
\EM{\textrm{clMap}(#1, #2)}
}

\DeclareDocumentCommand{\CHP}{d()}
{
\IfNoValueTF{#1}
{
    \textrm{(CHP)}
}
{
    \textrm{(\EM{#1}-CHP)}
}
}

\DeclareDocumentCommand{\CJEP}{d()}
{
\IfNoValueTF{#1}
{
    \textrm{(CJEP)}
}
{
    \textrm{(\EM{#1}-CJEP)}
}
}

\makeatletter
\newcommand{\dotminus}{\mathbin{\text{\@dotminus}}}

\newcommand{\@dotminus}{%
  \ooalign{\hidewidth\raise1ex\hbox{.}\hidewidth\cr$\m@th-$\cr}%
}

\def\w{\EM{\omega}}

\def\^{\EM{{}^{\And}}}

\def\And{\EM{\wedge}}

\def\<{\EM{\langle}}
\def\>{\EM{\rangle}}

\def\nl{\newline}

\def\EM#1{\ensuremath{#1}}

\def\mbf#1{\EM{\mathop{\pmb{#1}}}}

\def\ul#1{\underline{#1}}

\def\st{\,:\,}
\def\:{\colon}

\providecommand{\dotdiv}{
  \mathbin{
    \vphantom{+}
    \text{
      \mathsurround=0pt 
      \ooalign{
        \noalign{\kern-.35ex}
        \hidewidth$\smash{\cdot}$\hidewidth\cr 
        \noalign{\kern.35ex}
        $-$\cr 
      }%
    }%
  }%
}

\DeclareDocumentCommand{\RightJustify}{m}{\hspace*{\fill}\mbox{#1}\penalty-9999\relax}

\DeclareDocumentCommand{\DeclareCounter}{m}%
{%
 	\ifcsname c@#1\endcsname%
	\else%
		\newcounter{#1}%
	\fi%
}

\DeclareDocumentCommand{\MyQED}{}{\qed}

\DeclareDocumentEnvironment{proof}{d() D[]{\MyQED} d<>}
{
\noindent\IfNoValueTF{#1}
{\emph{Proof.\!\!}}
{\emph{Proof\ #1.\ }}
}
{
\IfValueTF{#3}
	{%
	\ExplSyntaxOff
	\RightJustify{\EM{#2_{#3}}}
	\vspace{1ex}
	}
	{
	\RightJustify{\EM{#2}}
	\vspace{1ex}
	}
}

\DeclareCounter{ProofLabelcOUntEr}
\DeclareDocumentCommand{\ProofLabel}{}{%
\addtocounter{ProofLabelcOUntEr}{1}
\label{cUrrEntProoflAbEl\arabic{ProofLabelcOUntEr}}
}

\DeclareDocumentCommand{\ProofRef}{D<>{1}}
{%
\ref{cUrrEntProoflAbEl\arabic{ProofcOUntEr#1}}
}
\DeclareDocumentCommand{\ProofCref}{D<>{1}}
{%
\cref{cUrrEntProoflAbEl\arabic{ProofcOUntEr#1}}
}

\DeclareDocumentEnvironment{Proof}{d() D[]{\MyQED} D<>{1}}
{
\DeclareCounter{ProofcOUntEr#3}%
\setcounter{ProofcOUntEr#3}{\value{ProofLabelcOUntEr}}%
\begin{proof}(#1)[#2]<\ProofRef<#3>>%
}
{
\end{proof}
}

\ExplSyntaxOn
\def\TheoremDepth{section}

\DeclareDocumentCommand{\DeclareTheorem}{m o m o}{%
\IfNoValueTF{#4}
	{%
	\IfNoValueTF{#2}
		{%
		\newtheorem{#1vArIAblE}{#3}
		}
		{%
		\newtheorem{#1vArIAblE}[#2vArIAblE]{#3}
		}
	}
	{%
	\newtheorem{#1vArIAblE}{#3}[#4]%
	}
\newtheorem*{#1vArIAblE*}{#3}

\DeclareDocumentEnvironment{#1}{o o}

	{
	\IfValueT{##2}%
		{
		\begin{spacing}{##2}
		}
	\IfValueTF{##1}
		{
		\begin{#1vArIAblE}[##1]
		}
		{
		\begin{#1vArIAblE}
		}
	\ProofLabel
	}
	{
	\IfValueT{##2}%
		{
		\end{spacing}{##2}
		}
	\end{#1vArIAblE}
	}

\DeclareDocumentEnvironment{#1*}{o o}

	{
	\IfValueT{##2}%
		{
		\begin{spacing}{##2}
		}
	\IfValueTF{##1}
		{
		\begin{#1vArIAblE*}[##1]
		}
		{
		\begin{#1vArIAblE*}
		}
	}
	{
	\IfValueT{##2}%
		{
		\end{spacing}{##2}
		}
	\end{#1vArIAblE*}
	}
}
\ExplSyntaxOff

\theoremstyle{plain}
\DeclareTheorem{theorem}{Theorem}[\TheoremDepth]
\DeclareTheorem{proposition}[theorem]{Proposition}
\DeclareTheorem{lemma}[theorem]{Lemma}
\DeclareTheorem{corollary}[theorem]{Corollary}
\DeclareTheorem{claim}[theorem]{Claim}

\theoremstyle{definition}
\DeclareTheorem{definition}[theorem]{Definition}
\DeclareTheorem{axiom}[theorem]{Axiom}
\DeclareTheorem{convention}[theorem]{Convention}

\theoremstyle{remark}
\DeclareTheorem{remark}[theorem]{Remark}
\DeclareTheorem{example}[theorem]{Example}
\DeclareTheorem{question}[theorem]{Question}
\DeclareTheorem{conjecture}[theorem]{Conjecture}

\begin{document}

\title{Computability of Countable Sunflowers}

\begin{abstract}
We provide a characterization of when a countably infinite set of finite sets contains an infinite sunflower. We also show that the collection of such sets is Turing equivalent to the set of programs such that whenever the program converges it returns the code of a program with finite range. 
\end{abstract}

\author{Nathanael Ackerman}
\address{Harvard University,
Cambridge, MA 02138, USA}
\email{nate@aleph0.net}

\author{Leah Karker}
\address{Providence College, RI 02918 USA}
\email{mkarker@providence.edu}

\author{Mostafa Mirabi}
\address{The Taft School, Watertown, CT 06795, USA}
\email{mmirabi@wesleyan.edu}



\maketitle

\section{Introduction}

In combinatorial set theory a \emph{sunflower}, or a \emph{$\Delta$-system}, is a collection of sets any two pairs of which have a common intersection. The existence of large sunflowers holds significance in various fields, including the study of circuit lower bounds, matrix multiplication, pseudo-randomness, and cryptography. For an overview of the connections to computer science see \cite{MR4334977}.

Furthermore, the existence of large sunflowers also has applications in mathematical logic, including in the study of forcing large generic structures. For further insights, you may refer to works such as \cite{Golshani}, \cite{Cohen}, and \cite{Cohen-Generic-with-Functions_AGM}. 

Among the most important results regarding sunflowers are \Erdos\ and Rado's ``Sunflower Lemma'' and the ``$\Delta$-systems Lemma'' of set theory.

\begin{lemma}[Sunflower Lemma \cite{Erdos-Rado}]
\label{Sunflower lemma}
Let $k, n \in \w$ with $k \geq 3$. Any set $S \subseteq \Powerset_{n}(\w)$ with $|S| \geq k!(n-1)^k$ contains a sunflower of size $n$. 
\end{lemma}

\begin{lemma}[$\Delta$-System Lemma]
Suppose $B \subseteq \Powerset_{<\w}(\kappa)$ where $\kappa$ is an uncountable cardinal. Then there is a $\Delta$-system $B_0 \subseteq B$ with $|B_0| = \kappa$. 
\end{lemma}

Both of these lemmas show that, provided that we have a ``large enough'' collection of finite sets we must have a ``large'' sunflower. The Sunflower Lemma shows that, provided all sets have the same size, any sufficiently large finite collection of sets must have a large finite sunflower. The $\Delta$-Systems Lemma shows that any uncountable collection of finite sets must have an uncountable sunflower. 

Together these results show that if our collection of sets is large and finite or uncountable then it must contain a sunflower. However, these results do not address the scenario when the collection of finite sets is countable. And in fact it is easy to see that the collection of countable ordinals $\{[n]\}_{n \in \w}$ is a countably infinite set with no sunflower of size $3$. 

In this paper, we consider the question ``When does a countably infinite collection of finite sets have an infinite sunflower?''. Specifically, for any infinite collection of finite sets and for each $n \in \w$ we can define a tree of sunflowers which contain a set of size $n$. We will then show that the collection contains a sunflower if and only if one such tree is infinite. 

In the course of characterizing those collections of finite sets which contain a sunflower we will also give a computability theoretic bound on how complicated a sunflower must be.

\subsection{Notation}

We let $\Powerset_{<\w}(X)$ be the collection of finite subsets of $X$. For $n \in \w$ we let $\Powerset_{n}(X)$ be the collection of subsets of $X$ of size $n$. If $f$ is a function we let $\dom(f)$ denote its domain and $\range(f)$ denote its range. 

If $X$ is a computable structure, e.g. $\w$, $\w^n$, a computable subset of $\w$, etc. then we fix encodings of $\Powerset_{<\w}(X)$ such that from an encoding for a set $Y$ both the relation $y \in Y$ and $|Y| = n$ are computable for $n\in \w$ and $y \in X$. We say a partial function $f$ \defn{encodes} a subset of $\Powerset_{<\w}(X)$ if the domain of $f$ is a subset of $\w$ and the range of $f$ is contained in $\Powerset_{<\w}(X)$. 

We fix a computable enumeration of the partial computable functions and for $e \in \w$ we let $\{e\}$ be the $e$th partial computable function. Suppose $X \subseteq \w$. We let $\TuringJump[X]$ be the Turing jump of $X$. If $X, Y \subseteq \w$ we say $X \TuringEq Y$ if $X$ and $Y$ have the same Turing degrees and $X \TuringLeq Y$ if $X$ is Turing reducible to $Y$.

We let 
\[
\FinSetComp[X] = \{e \in \w \st \range(\{e\}^X)\text{ is finite}\}, 
\]
i.e. the collection of programs with oracle $X$ whose range is finite. 

If $X \subseteq \w$ we let 
\[
\FinSetSeqComp[X] = \Big\{e \in \w \st (\forall n \in \w)\, \Big(\{e\}^X(n)\! \downarrow \text{ implies } \{e\}^X(n) \in \FinSetComp[X]\Big)\Big\}.
\]

Intuitively, $e \in \FinSetSeqComp[X]$ if $e$ outputs, with oracle $X$, codes for a sequence of functions and, for each output, the range of the function is finite. In particular by letting $e$ output the constant value $e^*$ we can see that $\FinSetComp[X]$ is computable from $\FinSetSeqComp[X]$. 

Note that $\FinSetComp[X]$ is a $\Sigma^0_2(X)$ set and so $\FinSetSeqComp[X]$ is a $\Pi^0_3(X)$ set.

\begin{definition}
Suppose $X$ is a set. A \defn{sunflower} on $X$ is a subset $X_0 \subseteq \Powerset_{<\w}(X)$ such that 
\[
(\exists r)\, (\forall p, q \in X)\, p \neq q \rightarrow p \cap q = r.
\]
We say a set $X$ \defn{contains} a sunflower if there is an $X_0 \subseteq X$ which is a sunflower. 
\end{definition}

\section{Constant Set Size}

We now consider countably infinite sunflowers of sets of a fixed size. We first note that from a computability standpoint considering sets of a bounded size and considering sets of a constant size are essentially the same. 

\begin{lemma}
Suppose $X \subseteq \Powerset_{<\w}(Y)$ is such that 
\[
(\exists n\in \w)(\forall x \in X)\, |x| \leq n. 
\]
Then there is a set $X^* \subseteq \Powerset_{n}(Y \times \{0\} \cup \w \times \{1\})$ and a bijection $i\:X \to X^*$ such that 
\begin{itemize}
\item $X^*$ and $i$ are computable from $X$, 

\item For any $x \in X$, $\{(a, 0) \st a \in x\} =  i(x) \cap (Y \times \{0\})$, 

\item For all $X_0 \subseteq X$, $X_0$ is a sunflower if and only if $i``[X_0]$ is a sunflower and in particular for any $x, y \in X$, $i(x) \cap i(y) = \{(a, 0) \st a \in x \cap y\}$. 

\end{itemize}
\end{lemma}
\begin{proof}
Note that $Y \times \{0\} \cup \w \times \{1\}$ is the disjoint union of $Y$ and $\w$. With this intuition we enumerate the elements of $X$ and for each element $x \in X$ with $|x| < n$ we add $n - |x|$ new elements from $\w$ to $i(x)$. Further we add elements such that if $x, y \in X$ are distinct then no elements added to $i(x)$ are also added to $i(y)$. 
\end{proof}

We now turn our attention to infinite subsets of $\Powerset_{n}(\w)$.

\begin{theorem}
\label{Sunflowers on sets of the same size}
Suppose 
\begin{itemize}
\item $n \in \w$, 

\item $S \subseteq \Powerset_{n}(\w)$ is infinite.
\end{itemize}
Then there is a sunflower $S_0 \subseteq S$ which infinite is computable in $\FinSetComp[S]$. 

Furthermore there is an $e \in \w$ such that for any encoding $Z$ of $\FinSetComp[S]$, $\{e\}^{Z}$ is a total function which is an encoding of a sunflower contained in $S$. 
\end{theorem}
\begin{proof}
For $A \in \Powerset_{<\w}(\w)$ let $S_A = \{s \in S \st A \subseteq s\}$. We prove the result by induction on $n$. First note that if $n = 1$ then $S$ is a sunflower and so we are done.  

Now suppose the theorem is true for $n$ and $S \subseteq \Powerset_{ n+1}(\w)$. We break into two cases. \nl\nl
\ul{Case 1:} $(\exists a \in \w)|S_{\{a\}}| = \w$ \nl
Let $S^* = \{s\setminus \{a\} \st s  \in S_{\{a\}}\}$. Then $S^* \subseteq \Powerset_{n}(\w)$. We can therefore find an infinite sunflower $S_0^* \subseteq S^*$ which is computable in $\FinSetComp[S^*]$. But $S^*$ is computable from $S$ and so $\FinSetComp[S^*]$ is computable from $\FinSetComp[S]$. Therefore $S_0^*$ is also computable from $\FinSetComp[S]$. Let $S_0 = \{s \cup \{a\} \st s \in S_0^*\}$. Then $S_0$ is the desired infinite sunflower. \nl\nl
\ul{Case 2:} $(\forall a \in \w)|S_{\{a\}}| < \w$ \nl
For any finite $Q_0 \subseteq S$ there are only finitely many $X \in S$ such that $X \cap \bigcup Q_0 \neq \emptyset$. Let $I(Q_0)$ be the collection of all such $X$. Note that $I(Q_0)$ is uniformly computable in $\TuringJump[S]$ from $Q_0$ and $\TuringJump[S]$ is computable from $\FinSetComp[S]$. 

We now define our sunflower $S_0 = \{X_i\}_{i \in \w}$ by induction. First let $X_0 \in S$ be arbitrary. Then, for $k \in \w$, let $X_{k+1} \in S \setminus I(\{X_i\}_{i \in [k+1]})$. We then have for any $i < j \in \w$ that $X_i \cap X_j = \emptyset$ and so $S_0$ is an infinite sunflower. Further $S_0$ is computable from $\FinSetComp[S]$. \nl\nl
Finally note that we can identify from $\FinSetComp[X]$ which case we are in. Therefore we can compute $S_0$ from $\FinSetComp[S]$. 
\end{proof}

\section{Non-Constant Set Size}
We now consider the case when the size of the finite sets can be unbounded. The following definition will be important.

\begin{definition}
Suppose $S \subseteq \Powerset_{<\w}(\w)$. Let $T_{n, k}(S)$ be collection of finite sunflowers $S_0 \subseteq S$ such that $\min \{|s| \st s \in S_0\} \leq n$ and $\max \{|s| \st s \in S_0\} \leq n+k$. 

We then let $T_{n}(S)$ be the tree with ordering inclusion where level $k$ of $T_n(S) = T_{n, k}(S)$. We call this the \defn{$n$-th sunflower tree} of $S$. 
\end{definition}

Note that $T_n(P)$ is not simply the collection of sunflowers that contain an element of size $n$ ordered by inclusion, i.e. the elements of level $k$ are not all sunflowers of size $k$. \cref{Sunflowers on sets of the same size} tells us that, when trying to identify those collections $P$ with an infinite sunflower the only interesting case is when for each $n$ there are only finitely many sets of size $n$. As we will see, in this case we want to construct for each $n$, a tree that is finitely branching whose elements are sunflowers with an element of size $n$. However, if we were to consider the tree where the sunflowers at level $k$ were those of size $k$ then the trees would be infinitely branching. 

\begin{theorem}
\label{Sunflowers on sets of different sizes}
Suppose $X \subseteq \w$. There is a map $i\:\w \to \w$ which is computable in $\FinSetSeqComp[X]$ such that 
\begin{itemize}
\item[(a)] if $\{e\}^X$ does not encode an infinite subset of $\Powerset_{<\w}(\w)$ then $i(e) = 0$, 

\item[(b)] if $\{e\}^X$ encodes an infinite subset $\Powerset_{<\w}(\w)$ which does not contain a sunflower then $i(e) = 1$, 

\item[(c)] if $\{e\}^X$ encodes an infinite subset $S \subseteq \Powerset_{<\w}(\w)$ with a sunflower then $i(e) = e^* + 2$ where $\{e^*\}^{\FinSetSeqComp(X)}$ encodes a countable sunflower $S_0 \subseteq S$. 

\end{itemize}
\end{theorem}
\begin{proof}
Suppose $\{e\}^X$ encodes $S \subseteq \Powerset_{<\w}(\w)$. If $S$ is a finite subset we can determine this from $\FinSetComp[X]$. Therefore we can determine from $\FinSetComp[X]$ whether or not $i(e) = 0$. 

Note there is an $r$ such that for all $n \in \w$, $\{r\}^S(n)$ encodes $S \cap \Powerset_{n}(\w)$. We now break into cases, where which case we are in can be identified from $\FinSetSeqComp[X]$. \nl\nl
\ul{Case 1:} $(\exists n \in \w)\, S \cap \Powerset_{n}(\w)$ is infinite. \nl
From $\FinSetComp[X]$ we can find one such $n$. We then have $i(e) > 1$ and we can use \cref{Sunflowers on sets of the same size} to find $e^*$ which is a code for a sunflower contained in $S$.\nl\nl 
\ul{Case 2:} $(\forall n \in \w)\, S \cap \Powerset_{n}(\w)$ is finite.\nl
We then have the following. 
\begin{itemize}
\item $\{T_n(S)\}_{n \in \w}$ is computable from $S$, 

\item $T_n(S)$ is finitely branching, as there are only finitely many elements of $P$ with sizes in $[n, n+k]$,

\item if $(Z_i)_{i \in \w}$ is an infinite branch of $T_n(S)$ then $\bigcup_{i \in \w} 
Z_i$ is an infinite sunflower contained in $S$, 

\item if $\{Y_i\}_{i \in \w}$ is an infinite sunflower contained in $S$ and $n = \inf\{|Y_i|\}_{i \in \w}$ then there is an infinite branch $(Z_i)_{i \in \w}$ in $T_n(S)$ such that $\bigcup_{i \in \w} Z_i = \{Y_i\}_{i \in \w}$, 
\end{itemize}

In particular $S$ contains an infinite sunflower if and only if for some $n \in \w$, $T_n(S)$ is infinite. But we can determine this from $\FinSetSeqComp[X]$. Therefore we can determine from $\FinSetSeqComp[X]$ whether or not $i(e) = 1$. 

Further if $S$ does contain a sunflower, i.e. $i(e) > 1$, we can find a $T_n(S)$ which is infinite from $\FinSetComp[X]$. Then we can construct an infinite branch in $T_n(S)$ from $\FinSetComp[X]$ as an element of $T_n(S)$ is part of an infinite branch if and only if it has infinitely many children (as $T_n(S)$ is finitely branching). From such an infinite branch we then can construct our desired infinite sunflower. All of this can be done from $\FinSetComp[X]$ and so we can find an $e^*$ such that $\{e^*\}^{\FinSetSeqComp[X]}$ encodes an infinite sunflower. We then let $i(e) = e^* + 2$. 
\end{proof}

\begin{corollary} 
Suppose $S\subseteq \Powerset_{<\w}(\w)$. Then the following are equivalent. 
\begin{itemize}
\item[(a)] for some $n \in \w$, $T_n(S)$ is infinite, 

\item[(b)] $S$ contains an infinite sunflower. 

\end{itemize}
\end{corollary}
\begin{proof}
We break into two cases. \nl\nl
\ul{Case 1:} For some $n \in \w$ we have $|S \cap \Powerset_{n}(\w)| = \w$. \nl
In this case $T_n(S)$ is infinite and so (a) holds. Further, by \cref{Sunflowers on sets of the same size} we have (b) holds. \nl\nl
\ul{Case 2:} For all $n \in \w$ we have $|S \cap \Powerset_{n}(\w)| <  \w$. \nl
In this case for all $n \in \w$ we have $T_n(S)$ is finitely branching for all $n \in \w$. Therefore (a) holds if and only if for some $n$ there is an infinite path through $T_n(S)$. But there is an infinite path through $T_n(S)$ if and only if there is an infinite sunflower $S_0 \subseteq S$ with $\min \{|s|\st s \in S_0\} \leq n$. Therefore (a) implies (b). 

Finally suppose (b) holds. Suppose $S_0 \subseteq S$ is an infinite sunflower and let $n = \min \{|s| \st s \in S_0\}$. Then $S_0$ gives rise to an infinite path though $T_n(S)$. But this implies $T_n(S)$ is infinite and so (a) holds. 
\end{proof}

\section{Tightness}

We have shown how to determine from $\FinSetSeqComp[X]$ whether or not a set computable in $X$ has an infinite sunflower. We now show the converse. 

\begin{definition}
Suppose $X \subseteq \w$. Let $$\SunflowerComp[X] = \left\{e \st \{e\}^X\text{ encodes an infinite set containing an infinite sunflower}\right\}.$$ 
\end{definition}

\begin{theorem}
For $X \subseteq \w$, $\SunflowerComp[X] \TuringEq \FinSetSeqComp[X]$. Further for each direction there is a single program which is independent of $X$ which witnesses the reduction. 
\end{theorem}
\begin{proof}
That $\SunflowerComp[X] \TuringLeq \FinSetSeqComp[X]$ follows from \cref{Sunflowers on sets of different sizes}. 

We now show how to compute $\FinSetSeqComp[X]$ from $\SunflowerComp[X]$. Suppose $e \in \w$. For $n, m \in \w$ let 
\[
E(n, m) = \{(0, i)\}_{i \in [n+1]} \cup \{(m, n)\}
\]
and let $E^*(n, m) = E(n, \{\{e\}^X(n)\}^X(m))$ when $\{\{e\}^X(n)\}^X(m) \downarrow$.  Let $B_e \subseteq \Powerset_{< \w}(\w \times \w)$ consist of all sets of the form $E(n, 0)$ and $E^*(n, m)$ where $\{\{e\}^X(n)\}^X(m) \downarrow$. Note uniformly in $X$ and $e$ we can computably find an $e^*$ so that $\{e^*\}^X$ encodes $B_e$. It therefore suffices to prove the following claim.

\begin{claim}
For $e \in \w$, $e \in \FinSetSeqComp[X]$ if and only if $B_e$ does not have an infinite sunflower. 
\end{claim}
\begin{proof}
If $e \not \in \FinSetSeqComp[X]$ then for some $n$, $\{e\}^X(n) \downarrow$ and $\{e\}^X(n)$ has infinite range. But then there are infinitely many sets of the form $E^*(n,m)$ in $B_e$ which collectively form a sunflower.  

Now suppose $e \in \FinSetSeqComp[X]$ to show $B_e$ has no infinite sunflower. Note for any $E(n, m), E(n', m') \in B_e$ we have 
\[
E(n, m) \cap E(n', m') = E(\min\{n,n'\}, 0).
\]
Now suppose $B^* \subseteq B_e$ is a sunflower. Then there is an $r \in \w$ such that for all $E(n_0, m_0), E(n_1, m_1) \in B^*$, $\min\{n_0, n_1\} = r$. But this implies that $|\{E(n, m) \in B_e \st n \neq r\}| = 1$. Therefore $|B^*| \leq |\range(\{\{e\}^X(n)\}^X)| + 1$ which is finite. 
\end{proof}
\end{proof}

\bibliographystyle{amsnomr}
\bibliography{bibliography}

\end{document}